\documentclass[a4paper,12pt,onecolumn,twoside]{article}
\usepackage{mathrsfs}
\usepackage{fancyhdr}
\usepackage{amsmath,amsfonts,amssymb,graphicx,amsthm}
\allowdisplaybreaks
\usepackage{color}
\usepackage{upgreek} % 希腊字母转正体, 命令为 \upalpha, \upbeta
\usepackage{hyperref} %“超链接引擎”
\usepackage{subfigure} %%图形或表格并排排列
\newcommand{\pll}{\kern 0.56em/\kern -0.8em /\kern 0.56em}  % 定义平行
\usepackage{indentfirst} % latex默认section后首行无缩进，这个宏包可以缩进
\usepackage{tikz} 
\usetikzlibrary{math,calc,intersections,through,backgrounds,animations,arrows,patterns,decorations.pathmorphing,decorations.pathreplacing,decorations.shapes,quotes,angles}  %Tikz 程序库，可选
\tikzset{global scale/.style={
    scale=#1,
    every node/.append style={scale=#1}
  }
}  % 在 tikzpicture 环境或者 \tikz 命令的参数中，可以通过 scale 选项来缩放绘制的图形。然而，这种缩放不会同步应用在 node 中的文字上，导致图形与 node 中的文字注解大小失衡。 因此定义一个global scale

\numberwithin{equation}{section}

 \newtheorem{lemma}{Lemma}[section]
 \newtheorem{proposition}[lemma]{Proposition}
 \newtheorem{theorem}{Theorem}[section]
 
 \theoremstyle{remark}
 \newtheorem{remark}{Remark}[section]

\numberwithin{equation}{section}

\begin{document}

%--------------------------------------------------------------------------------------
%--------------------------------------------------------------------------------------
\title{\bf Gradient Catastrophe for Solutions to the Hyperbolic Navier-Stokes Equations}

%\author[a,b]{Huijiang Zhao\thanks{Email address: hhjjzhao@hotmail.com}}
%\author[a]{Qingsong Zhao\thanks{Corresponding author. Email: qszhao@whu.edu.cn}}
%\affil[a]{\small School of Mathematics and Statistics, Wuhan University, Wuhan 430072, China}
%\affil[b]{\small Computational Science Hubei Key Laboratory, Wuhan University, Wuhan 430072, China}
%

\author{{\sc Qingsong Zhao}\thanks{School of Mathematics and Science, Nanyang Institute of Technology, Nanyang 473004, China. Email: qqsszhao@nyist.edu.cn}}

\date{}

\maketitle
\begin{abstract}
This paper studies local existence and the singularity formation of the solutions of the one-dimensional hyperbolic Navier–Stokes equations, in particular proving the gradient blow-up of the derivatives of the solutions. The underlying model introduces a relaxation mechanism that leads to hyperbolization, achieved both through a nonlinear Cattaneo law for heat conduction and through Maxwell‑type constitutive relations for the stress tensor. Our main approach is to prove that the hyperbolic Navier–Stokes equations are indeed hyperbolic, and to prove that they possess two genuinely nonlinear eigenvalues, thereby establishing the blow-up of the gradient of the solution. In addition, we provide a derivation of the equation of state for the hyperbolic Navier–Stokes equations in the appendix.
\\[2mm]
\\[2mm]
\end{abstract}

\maketitle

\tableofcontents

%---------------------------------------------------------------------------------------
\section{Introduction}

The fundamental governing equations of fluid mechanics—the Navier-Stokes equations—are of fundamental importance in science and engineering. However, the classical form of these equations exhibits a well-known physical inconsistency: they predict that disturbances in pressure and temperature propagate at an infinite speed \cite{Abdelhedi}. This unphysical characteristic stems from their underlying constitutive relations, particularly Fourier's law of heat conduction. As a parabolic equation, Fourier's law implies an instantaneous relationship between the heat flux response and the temperature gradient.

To incorporate the finite-speed signal propagation observed in the physical world into mathematical models, researchers have proposed various modifications. A prominent approach involves introducing a `relaxation time' ($\tau$). This replaces the algebraic constitutive equation (parabolic type) with a differential equation containing a time derivative, thereby transforming the system into a hyperbolic or mixed parabolic-hyperbolic type (see \cite{Abdelhedi,Coulaud}). The modified system is often referred to as the hyperbolic Navier-Stokes equations or hyperbolized Navier-Stokes equations. In the one-dimensional case, these equations are given by 
\begin{subequations}\label{HNS}
\begin{align}
v_t-u_x&=0,  \label{HNS-a}\\
u_t+p_x&=S_x, \label{HNS-b}\\
\left(e+\frac12u^2\right)_t+(up)_x+q_x&=(uS)_x, \label{HNS-c}\\
%e_t+pu_x+q_x&=Su_x, \label{HNS-c}\\
\tau_1 q_t+q+\frac{\kappa\theta_x}{v}&=0, \label{HNS-d}\\
\tau_2 S_t+S&=\frac{\mu u_x}{v}. \label{HNS-e}
\end{align}
\end{subequations}
Here $x\in \mathbb{R}$ is the Lagrangian space variable, $t\in \mathbb{R}^+$ is the time variable. The primary dependent variables are specific volume $v$, fluid velocity $u$, absolute temperature $\theta,$ stress tensor $S$ and heat flux $q$. The heat conductivity coefficient $\kappa$ and viscous coefficient $\mu$ are positive constants. The positive constants $\tau_1$ and $\tau_2$ are relaxation parameters.

In the hyperbolic Navier-Stokes equations \eqref{HNS}, the first three equations \eqref{HNS-a}-\eqref{HNS-c} are conservation laws, encompassing mass conservation, momentum conservation, and energy conservation. To overcome the paradox of infinite propagation speed, Cattaneo proposed a revolutionary idea in the mid-20th century: introducing a tiny relaxation time $\tau_1$ to modify Fourier's law $q=-\frac{\kappa\theta_x}{v}$ into a relaxation equation that includes a time derivative, namely Cattaneo's law \eqref{HNS-d}. Similarly, in terms of fluid viscosity, the Maxwell fluid model \eqref{HNS-e} modifies the Newtonian fluid model $S=\frac{\mu u_x}{v}$ by introducing a relaxation process for the stress tensor $S$.

 The five thermodynamic quantities: specific internal energy $e$, specific volume $v$, absolute temperature $\theta$, entropy $s$, and pressure $p$ are related by the Gibbs equation
\begin{equation}\label{Gibbs}
 \theta\mathrm{d}s=\mathrm{d}e+p\mathrm{d}v.
\end{equation}

We assume that the pressure $p$ and the specific internal energy $e$, as a function of $(v,\theta,q,S),$ satisfy
\begin{align}
p(v,\theta,q,S)&=\tilde{p}(v,\theta)-\frac{\tau_1}{2\kappa\theta}q^2-\frac{\tau_2}{2\mu}S^2, \label{p-}\\
e(v,\theta,q,S)&=\tilde{e}(v,\theta)+\frac{\tau_1 v}{\kappa\theta}q^2+\frac{\tau_2v}{2\mu}S^2, \label{e-}
\end{align}
such that 
\begin{equation}\label{gibbs}
e_v(v,\theta,q,S)=\theta(v,\theta,q,S) p_\theta(v,\theta,q,S)-p(v,\theta,q,S)
\end{equation}
holds. Here \eqref{gibbs} follows from \eqref{Gibbs}. The functions $\tilde{p}(v,\theta)$ and $\tilde{e}(v,\theta)$ are the classic expressions of pressure and internal energy as functions of $(v,\theta)$ with relaxation parameters $\tau_1=\tau_2=0.$ In Appendix B, we will provide a detailed interpretation of the expressions $p(v,\theta,q,S)$ in \eqref{p-} and $e(v,\theta,q,S)$ in \eqref{e-}. Interested readers may also refer to reference \cite{Chen-ZAMP-1970,Coleman-ARMA-1986,Coleman-ARMA-1986-2,Tarabek-QAM-1992}.
Furthermore, since under general conditions, pressure and density are positively correlated in a closed system at constant temperature, and internal energy and temperature are positively correlated under constant volume, we have hypothesis
\begin{equation}\label{hypothesis}
p_v(v,\theta,q,S)<0,\quad\quad\quad e_\theta(v,\theta,q,S)>0.
\end{equation}

For the hyperbolic Navier-Stokes equations \eqref{HNS}, what is particularly striking is how the blow-up result sharply contrasts with the relaxation‑free setting ($\tau_1=\tau_2=0$). In the classical compressible Navier–Stokes system ($\tau_1=\tau_2=0$), global large solutions are known to exist (see Kazhikhov \cite{Kazhikhov-JAMM-1977,Kazhikhov-SMJ-1982}, Jiang \cite{Jiang-1999,Jiang-2002}, Li and Liang \cite{Li-Liang-2016}), making the contrast all the more significant. The Cauchy problem for system \eqref{HNS} and its multi‑dimensional analogue has received extensive attention. For the case $\tau_1=0, \tau_2\neq 0,$ for which the Maxwell fluid model \eqref{HNS-e} holds, the isentropic Navier–Stokes equations with the revised linear Maxwell law  were first taken up by Yong [32]. His work provided a local well‑posedness theory and identified the local relaxation limit. This result was then extended by Hu and Racke \cite{Hu-JMFM-2017} to the non-isentropic case and by Peng \cite{Peng-PANL-2021} to a more general setting. Hu and Wang \cite{Hu-Wang-AML-2020,Hu-Wang-MN-2019} as well as B\"{a}rlin \cite{Barlin-2022} investigated the occurrence of blow‑up. For the case $\tau_1\neq 0, \tau_2\neq 0,$ for which the Cattaneo's law \eqref{HNS-d} and the Maxwell fluid model \eqref{HNS-e} hold, global existence of smooth solutions for small initial data was established in by Hu and Racke in \cite{Hu-JDE-2020}. Meanwhile, Hu, Racke and Wang in \cite{Hu-JDE-2022} obtained a blow‑up result for large data.

Blow-up results for fluid dynamics equations can generally be classified into two types:
\begin{enumerate}
\item [(i).] The gradient of the solution blows up in finite time (also referred to as ``gradient catastrophe"). The main idea is to choose appropriate initial data so that the gradient of the Riemann invariant blows up, see \cite{Hormander,John-JDE-1979,Wang-Chen-JDE-1998}. The advantage of this approach is that it reveals the blow-up mechanism, i.e., the formation mechanism of singularities in the solution;
\item [(ii).] Sideris-type blow-up results. The core method relies on the finite speed of propagation of solutions to hyperbolic conservation laws, where certain functionals are constructed to prove that the solution cannot be extended globally in time , see \cite{Hu-Racke-SIMA-2023,Hu-JDE-2022,Sideris-CMP-1985}. Compared with the first type of blow-up result, this approach does not reveal the mechanism of singularity formation.
\end{enumerate}

Since the blow-up results in \cite{Hu-Racke-SIMA-2023,Hu-JDE-2022} are of Sideris-type and do not provide the mechanism of singularity formation, a natural question arises: {\it For the hyperbolic Navier-Stokes equations \eqref{HNS}, can one establish a blow-up mechanism for non-vacuum solutions, i.e., prove that the gradient of the solution blows up?}

The main purpose of this paper is devoted to such a problem. We consider the Cauchy problem for the functions 
\begin{equation*}
(v,u,\theta,q,S):[0,+\infty)\times\mathbb{R}\longrightarrow \mathbb{R}_+\times \mathbb{R}\times\mathbb{R}_+\times\mathbb{R}\times\mathbb{R}
\end{equation*}
with prescribed initial data
\begin{equation}\label{initial}
\left(v(0,x),u(0,x),\theta(0,x),q(0,x),S(0,x)\right)=\left(v_0(x),u_0(x),\theta_0(x),q_0(x),S_0(x)\right).
\end{equation}
Furthermore, it is assumed to satisfy the following far‑field conditions
\begin{equation}\label{far-field}
\lim_{|x|\to\infty}\left(v_0(x),u_0(x),\theta_0(x),q_0(x),S_0(x)\right)=(v_\pm,u_\pm,\theta_\pm,q_\pm,S_\pm).
\end{equation}
This paper focuses on the case where the far‑field states of the initial data $\left(v_0(x),u_0(x),\theta_0(x),q_0(x),S_0(x)\right)$ are equal, i.e., 
\begin{equation}\label{far-field-}
(v_-,u_-,\theta_-,q_-,S_-)=(v_+,u_+,\theta_+,q_+,S_+).
\end{equation}
Without loss of generality, we may normalize these constants to $v_\pm=1,$ $u_\pm=0,$ $\theta_\pm=1,$ $q_\pm=0$ and $S_\pm=0$
throughout the remainder of this paper.

Now we turn to state our main results. 

\begin{theorem}\label{thm:local}
Suppose that the pressure $p$ and the specific internal energy $e $ satisfy \eqref{p-}-\eqref{hypothesis}. Then for any $s\geq 2$, there exists a positive constant $\delta$ such that if $(v_0-1,u_0,\theta_0-1,q_0,S_0)\in W^{s,2}(\mathbb{R})$ and $\|(v_0-1,u_0,\theta_0-1,q_0,S_0)\|_{W^{s,2}(\mathbb{R})}<\delta$, then the Cauchy problem \eqref{HNS} with initial data \eqref{initial}--\eqref{far-field-} admits a unique local solution $(v,u,\theta,q,S)$ in some time interval $[0,T]$ with
\begin{equation*}
(v-1,u,\theta-1,q,S)\in C^0([0,T],H^2(\mathbb{R}))\cap C^1([0,T],H^{s-1}(\mathbb{R})).
\end{equation*}
\end{theorem}

Then we show the ``gradient catastrophe'' for hyperbolic Navier-Stokes equations \eqref{HNS}. We consider an ideal polytropic gas, i.e., the constitutive relations are
\begin{align}
p(v,\theta,q,S)&=R\frac{\theta}{v} -\frac{\tau_1}{2\kappa\theta}q^2-\frac{\tau_2}{2\mu}S^2, \label{p}\\
e(v,\theta,q,S)&=C_{\mathrm{v}}\theta+\frac{\tau_1 v}{\kappa\theta}q^2+\frac{\tau_2v}{2\mu}S^2, \label{e}
\end{align}
The adiabatic index $\gamma=1+\frac{R}{C_{\mathrm{v}}}$ of a gas directly reflects the ways in which the gas molecules store energy, i.e., the degrees of freedom. For monatomic gases, The adiabatic index $\gamma=\frac53.$ The more complex the molecular structure, the more ways there are to store energy, and the smaller the value of $\gamma$ becomes, gradually approaching $1.$ Hence, the adiabatic index $\gamma$ satisfies 
\begin{equation}\label{gamma}
1<\gamma\leq\frac{5}{3}.
\end{equation}
Our next result of this paper is stated as follows.
\begin{theorem}\label{thm:main}
Suppose that the pressure $p$ satisfies \eqref{p} and the specific internal energy $e$ satisfies \eqref{e}. Let the adiabatic index $\gamma $ satisfies \eqref{gamma}.
Then there exist functions $(v_0(x),$ $u_0(x),$ $\theta_0(x),$ $q_0(x),$ $S_0(x)),$ which satisfy \eqref{far-field} and \eqref{far-field-}, such that the unique $C^1-$solution of the hyperbolic Navier-Stokes equations \eqref{HNS} with initial data \eqref{initial} exists only for a finite time. More precisely, the solution of \eqref{HNS} with initial data \eqref{initial} remains bounded, while the first-order derivatives of the solution blow up in finite time.
\end{theorem}

\begin{remark}
H. Freist\"{u}hler in \cite{Freistuhler} proved the blow-up of solutions to a hyperbolic Navier-Stokes Equations, which consists of five equations. In his proof, the far-field values $v_\pm$ depend on the values of $\theta_\pm.$ We have removed this assumption. In our proof, it is irrelevant what values $v_\pm$ and $\theta_\pm$ take.
\end{remark}
 
Our approach is as follows. Let $\boldsymbol{u}=(v,u,\theta,q,S)^{\mathrm{T}}$. First, we write the hyperbolic Navier-Stokes functions \eqref{HNS} in the conservation laws form 
\begin{equation}\label{CLS}
\boldsymbol{u}_t+\boldsymbol{A}(\boldsymbol{u})\boldsymbol{u}_x={ \boldsymbol{g}(\boldsymbol{u})}
\end{equation}
and prove that  it is strictly hyperbolic near constant equilibrium state $\overline{\boldsymbol{u}}=(1,0,1,0,0)^\mathrm{T}$. Inspired by the work of B\"{a}rlin in\cite{Barlin}, which showed that
the solution of the conservation laws \eqref{CLS} with suitable initial data blows up when the following three assumptions hold:
\begin{enumerate}
\item [(H1).] the matrix $\boldsymbol{A}(\boldsymbol{0})$ has only real, simple eigenvalues;
\item [(H2).] Matrix $\boldsymbol{A}(\boldsymbol{0})$ possesses a genuinely nonlinear eigenvalue;
\item [(H3).] $\boldsymbol{g}(\boldsymbol{u})=\boldsymbol{0}.$
\end{enumerate}
Among these three assumptions, the truly core and difficult one is to prove that Matrix $\boldsymbol{A}(\boldsymbol{0})$ possesses a genuinely nonlinear eigenvalue. Since the velocity $u(t,x)$ doesn't appear in the expression of matrix $\boldsymbol{A}(\boldsymbol{u})$ and vector $\boldsymbol{g}(\boldsymbol{u}),$ the eigenvalues of $\boldsymbol{A}(\boldsymbol{u})$ can be seen as a function of $(v,\theta,q,S).$ From this perspective, proving that matrix $\boldsymbol{A}(\boldsymbol{0})$ has a genuinely nonlinear eigenvalue $\lambda_*$ does not seem difficult (because there appear to be only four independent variables $(v,\theta,q,S)$). However, we must not forget that the relaxation parameters $\tau_1$ and $\tau_2$, the heat conductivity coefficient $\kappa$, the viscous coefficient $\mu$, and the adiabatic index $\gamma$ here are all unknown; therefore, in fact, both matrix $\boldsymbol{A}(\boldsymbol{u})$ and its eigenvalues $\lambda_*(\boldsymbol{u})$ must be regarded as functions of the variables $(v,\theta,q,S,\tau_1,\tau_2,\mu,\kappa,\gamma)$:
\begin{align*}
\boldsymbol{A}(\boldsymbol{u})=&\boldsymbol{A}(v,\theta,q,S,\tau_1,\tau_2,\mu,\kappa,\gamma), \\
\lambda_*(\boldsymbol{u})=&\lambda_*(v,\theta,q,S,\tau_1,\tau_2,\mu,\kappa,\gamma).
\end{align*}
Although the universal gas constant $R$ is a given constant, we do not consider its value here either. In fact, we will later prove that the genuine nonlinearity of the eigenvalues is independent of the value of the universal gas constant $R$.

Second, we construct the Riemann invariant 
\begin{equation*}
\mathcal{R}(\tau_1,\tau_2,\mu,\kappa,\gamma)=\nabla_{\boldsymbol{u}}\lambda_*(\overline{\boldsymbol{u}})\cdot \boldsymbol{r}_*(\overline{\boldsymbol{u}})
\end{equation*}
near constant equilibrium state $\overline{\boldsymbol{u}}=(1,0,1,0,0)^\mathrm{T}$, here $\boldsymbol{r}_*(\overline{\boldsymbol{u}})$ is the eigenvector corresponding to the eigenvalue $\lambda_*(\overline{\boldsymbol{u}}).$ In the process of deriving the expression for the Riemann invariant $\mathcal{R}(\tau_1,\tau_2,\mu,\kappa,\gamma)$, we found that its expression can be simplified by introducing new variables
\begin{equation*}
w=\frac{\kappa}{\tau_1C_{\mathrm{v}}R},\quad\quad\quad z=\frac{\mu}{\tau_2R}.
\end{equation*}
In this way, the Riemann invariant $\mathcal{R}(\tau_1,\tau_2,\mu,\kappa,\gamma)$ can be expressed as a function of two variables $w$ and $z$ that involves the parameter $\gamma$: 
\begin{equation*}
\mathcal{R}(w,z;\gamma)=\mathcal{M}(w,z;\gamma)\mathcal{L}(w,z;\gamma)-w\mathcal{N}(w,z;\gamma).
\end{equation*}
Here parameter $\gamma\in \left(1,\frac53\right]$ is the adiabatic index. The definition of $\mathcal{L}(w,z;\gamma)$ is given is \eqref{L} and the definitions of $\mathcal{M}(w,z;\gamma)$, $\mathcal{N}(w,z;\gamma)$ are given in \eqref{M}, \eqref{N}.

Finally, we prove that the Riemann invariant 
\begin{equation*}
\mathcal{R}(z,w;\gamma)<0,\quad\quad\quad \forall z>0,\ \ \forall w>0,\ \  \forall \gamma\in \left(1,\frac53\right].
\end{equation*}
Our idea is to fix arbitrary $z\in (0,+\infty)$ and $\gamma\in (1,\frac53],$  and treat $w$ as the independent variable. Meanwhile, we find that the Riemann invariant $\mathcal{R}(w)$ has three expressions
\begin{equation*}
\mathcal{R}(w) = 
\begin{cases}
\mathcal{M}(w)\mathcal{L}(w)-w\mathcal{N}(w), & \text{(a)} \\
\frac12\left(\mathcal{Q}(w)-\mathcal{M}(w)g(w)\right) , & \text{(b)} \\
-\frac{2(\gamma-1)w^2\mathcal{P}(w)}{\mathcal{M}(w)g(w)+\mathcal{Q}(w)} . & \text{(c)}
\end{cases}
\end{equation*}
Here the functions $\mathcal{M}(w)$, $\mathcal{N}(w)$, $\mathcal{P}(w)$ and $\mathcal{Q}(w)$ are, respectively, a quadratic, a linear, and two cubic functions of the independent variable $w$. The function $g(w)>0$ is defined in \eqref{g}. Points $w_M(z),$ $w_N(z),$ $w_P(z)$ are the zeros of functions $\mathcal{M}(z)$, $\mathcal{N}(z)$, $\mathcal{P}(z)$ respectively, and point $w_Q(z)$ is the smallest zero of function $\mathcal{Q}(w).$ The relation of these four zeros are given in Lemma \ref{lem:roots} (see Figure \ref{fig:zeros}).
\begin{figure}[htpb]
\centering
\includegraphics[width=\columnwidth]{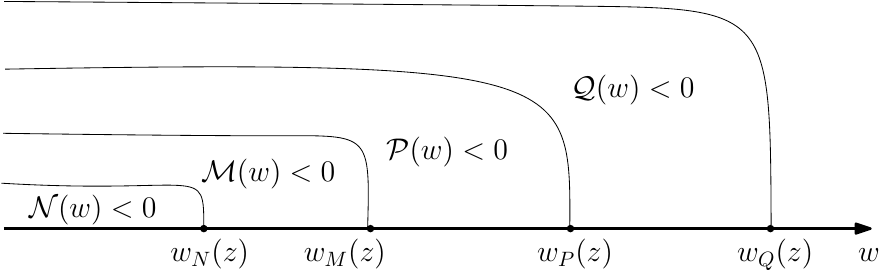} % 
\caption{The zeros of $w_M(z),$ $w_N(z),$ $w_P(z)$, $w_Q(z)$}
\label{fig:zeros}
\end{figure}
Now for any given $z\in (0,+\infty)$ and $\gamma\in (1,\frac53],$ we divide the interval $w \in (0, +\infty)$ into four parts
\begin{equation*}
(0,+\infty)=(0,w_N(z)]\cup (w_N(z),w_M(z)]\cup \mathcal{I}_-(z)\cup\mathcal{I}_+(z),
\end{equation*}
where
\begin{align*}
\mathcal{I}_-(z)=&\{w>w_M(z):\mathcal{Q}(w,z)\leq0\}, \\
\mathcal{I}_+(z)=&\{w>w_M(z):\mathcal{Q}(w,z)>0\}.
\end{align*}
To prove that the Riemann invariant $\mathcal{R}(w)<0,$ for the interval in which $w$ lies, use expression (c) of $\mathcal{R}(w)$ in the first and fourth parts, expression (a) in the second part, and expression (b) in the third part.

Our chapter arrangement is as follows. In section 2 we give the local existence of the solutions to the hyperbolic Navier-Stokes equations \eqref{HNS}. In section 3, we construct the Riemann invariant. The section 4 is devoted to proving that the Riemann invariant constructed above is negative, which implies that the hyperbolic Navier-Stokes equations \eqref{HNS} possesses a genuinely nonlinear eigenvalue. Appendix A is devoted to the property of $\mathcal{Q}(w),$ and Appendix B is devoted to the equation of state for hyperbolic Navier-Stokes equations \eqref{HNS}.

%\vspace{0.8em}
%\noindent\textit{\textbf{Notations.}} Throughout this paper, for given positive constant $T>0,$  for functions $f,g$ defined on $[0,T]\times \mathbb{R},$ we define  $f\lesssim g$ (or $f\gtrsim g$) if  there exists some constant $C$ depending only on $\Pi_0$ and $V_0$ such that  $f(t,x)\leq Cg(t,x)$ (or $f(t,x)\geq Cg(t,x)$) holds for any $(t,x)\in [0,T]\times \mathbb{R}$. Besides, we define $f\simeq g$ if $f\lesssim g$ and $f\gtrsim g.$

\section{Local existence}
In this section, we show that the system \eqref{HNS} is hyperbolic and then give the local existence result of the solutions to system \eqref{HNS}.

\subsection{Basic energy estimate}
In order to obtain the basic energy estimate for the hyperbolic Navier-Stokes equations, we first study the equation of entropy $s$. 

Using the Gibbs equation \eqref{Gibbs}, \eqref{gibbs} and the expressions for pressure $p(v,\theta,q,S)$ in \eqref{p-} and specific internal energy $e(v,\theta,q,S)$ in \eqref{e-}, the entropy $s$ satisfies
\begin{align*}
\frac{\partial s(v,\theta,q,S)}{\partial v}=&\frac{\partial p(v,\theta,q,S)}{\partial \theta}=\frac{\partial \tilde{p}(v,\theta)}{\partial\theta}+\frac{\tau_1}{2\kappa\theta^2}q^2 , \\
\frac{\partial s(v,\theta,q,S)}{\partial \theta}=&\frac{1}{\theta }\frac{\partial e(v,\theta,q,S)}{\partial \theta}=\frac{1}{\theta }\frac{\partial \tilde{e}(v,\theta)}{\partial  \theta } -\frac{\tau_1v}{\kappa\theta^3}q^2.
\end{align*}
Hence, the entropy $s$ satisfies
\begin{equation}\label{s}
s(v,\theta,q,S)=\tilde{s}(v,\theta)+\frac{\tau_1 v}{2\kappa\theta^2}q^2,
\end{equation}
where $\tilde{s}(v,\theta)$ yields
\begin{equation}\label{s_v-}
\frac{\partial\tilde{s}(v,\theta)}{ \partial v}=\frac{\partial \tilde{p}(v,\theta)}{\partial\theta},\quad\quad \frac{\partial\tilde{s}(v,\theta)}{\partial \theta } =\frac{1}{\theta} \frac{\partial \tilde{e}(v,\theta)}{\partial \theta}.
\end{equation}

Now we derive the equation of entropy $s$.
\begin{lemma}\label{lem:s}
The entropy $s$ satisfies
\begin{equation}\label{eq:entropy}
s_t+\left(\frac{q}{\theta} \right)_x=\frac{vq^2}{\kappa\theta^2}+\frac{vS^2}{\mu\theta}.
\end{equation}
\end{lemma}
\begin{proof}
Since the hyperbolic Navier-Stokes equations \eqref{HNS} consists of equations in terms of $v,$ $u$ and $e$, we take $(v, e)$ as the independent variables among the five thermodynamic quantities $(v, \theta, p, e, s)$, and regard the remaining three as functions of $(v, e)$. Denote $\theta=\bar{\theta }(v,e,q,S),$ then $s=\bar{s}(v,e,q,S)$ in \eqref{s} and $e$ in \eqref{e-} can be rewritten as 
\begin{align}
\bar{s}(v,e,q,S)=&\tilde{s}(v,\bar{\theta }(v,e,q,S))+\frac{\tau_1v}{2\kappa\bar{\theta }(v,e,q,S)^2}q^2, \label{s--} \\
e=&\tilde{e}(v,\bar{\theta }(v,e,q,S))+\frac{\tau_1v}{\kappa\bar{\theta }(v,e,q,S)}q^2+\frac{\tau_2v}{2\mu}S^2.  \label{e--}
\end{align}
From the Gibbs equation \eqref{Gibbs}, we have
\begin{align}
\frac{\partial \bar{s}(v,e,q,S)}{\partial v }=&\frac{p}{\theta },  \label{s_v}\\
\frac{\partial \bar{s}(v,e,q,S)}{\partial e}=&\frac{1}{\theta }. \label{s_e}
\end{align}
Due to 
\begin{equation}\label{s_t}
s_t=\bar{s}(v,e,q,S)_t=\frac{\partial \bar{s}}{\partial v }v_t+\frac{\partial \bar{s}}{\partial e }e_t+\frac{\partial \bar{s}}{\partial q }q_t+\frac{\partial \bar{s}}{\partial S }S_t,
\end{equation}
we need to calculate $\frac{\partial \bar{s}}{\partial q }$ and $\frac{\partial \bar{s}}{\partial S }$.

Take the partial derivatives of equation \eqref{e--} with respect to $q$ and $S,$ respectively. It holds that
\begin{align}
\left(\frac{\partial\tilde{e}(v,\bar{\theta }) }{\partial \theta }-\frac{\tau_1v}{\kappa\theta ^2}q^2\right)\frac{\partial\bar{\theta } }{\partial q} =&-\frac{2\tau_1v}{\kappa\theta }q,\label{theta_q}\\
\left(\frac{\partial\tilde{e}(v,\bar{\theta }) }{\partial \theta }-\frac{\tau_1v}{\kappa\theta ^2}q^2\right)\frac{\partial\bar{\theta } }{\partial S}=&-\frac{\tau_2v}{\mu }S.\label{theta_S}
\end{align}
Take the partial derivatives of equation \eqref{s--} with respect to $q$ and $S,$ respectively. By \eqref{s_v-}, \eqref{theta_q} and \eqref{theta_S}, we have
\begin{equation}\label{s_q}
\frac{\partial\bar{s}}{\partial q} =\left(\frac{\partial \tilde{s}}{\partial\theta }-\frac{\tau_1v}{\kappa\theta^3}q^2 \right)\frac{\partial\bar{\theta }}{\partial q}+\frac{\tau_1v}{\kappa\theta^2}q=-\frac{\tau_1v}{\kappa\theta^2}q
\end{equation}
and
\begin{equation}\label{s_S}
\frac{\partial\bar{s}}{\partial S} =\left(\frac{\partial \tilde{s}}{\partial\theta }-\frac{\tau_1v}{\kappa\theta^3}q^2 \right)\frac{\partial\bar{\theta }}{\partial S}=-\frac{\tau_2v}{\mu\theta}S.
\end{equation}
The equation \eqref{eq:entropy} is proved by substituting \eqref{s_v}, \eqref{s_e}, \eqref{s_q} and \eqref{s_S} into \eqref{s_t}, and using the hyperbolic Navier-Stokes equations \eqref{HNS}.
\end{proof}
 
The relative entropy $\eta(v,u,\theta,q,S)$ is a non-negative convex entropy, defined as follows:
\begin{align}
\eta(v,u,\theta,q,S)=&\left(e+\frac12u^2\right)-\left(s-\tilde{p} (1,1)(v-1)\right)-\tilde{e}(1,1)\nonumber\\
=&\tilde{\eta}(v,u,\theta)+\left(1-\frac{1}{2\theta} \right)\frac{\tau_1 v}{\kappa\theta}q^2+\frac{\tau_2v}{2\mu}S^2,\label{eta}
\end{align}
where 
\begin{equation*}
\tilde{\eta}(v,u,\theta)=\left(\tilde{e}(v,\theta) +\frac12u^2\right)-\left(\tilde{s}(v,\theta) -\tilde{p} (1,1)(v-1)\right)-\tilde{e}(1,1)
\end{equation*}
is the relative entropy with $\tau_1=\tau_2=0$.

Taking the derivative of $\eta(v,u,\theta,q,S)$ with respect to $t,$ using (\ref{HNS-a})-(\ref{HNS-c}), Lemma \ref{lem:s}, we have
\begin{equation*}
\eta_t+\left(u(p-\tilde{p}(1,1))+\left(1-\frac{1}{\theta}\right)q-uS\right)_x+\frac{vq^2}{\kappa\theta^2}+\frac{vS^2}{\mu\theta}=0.
\end{equation*}
Hence, we obtain  the basic energy estimate for the hyperbolic Navier-Stokes equation \eqref{HNS}.
\begin{lemma}\label{lem:basic energy}
The solutions $v(t,x),$ $u(t,x),$ $\theta(t,x),$ $q(t,x)$ and $S(t,x)$ of the hyperbolic Navier-Stokes equations \eqref{HNS} satisfy
\begin{multline*}
\int_{\mathbb{R}}^{} \eta(v,u,\theta,q,S) \mathrm{d}x  +\int_{0}^{t}\int_{\mathbb{R}}^{ } \left(\frac{vq^2}{\kappa\theta^2}+\frac{vS^2}{\mu\theta} \right)  \mathrm{d}x\mathrm{d}\tau   \\
=\int_{\mathbb{R}}^{} \eta(v_0,u_0,\theta_0,q_0,S_0) \mathrm{d}x,
\end{multline*}
here $\eta(v,u,\theta,q,S)$ is the relative entropy given in \eqref{eta}.
\end{lemma}

\subsection{Local existence}

For simplicity of notation, we use abbreviations for partial derivatives, such as 
\begin{equation*}
p_v=\frac{\partial p(v,\theta,q,S)}{\partial v},\quad\quad\quad e_\theta=\frac{\partial e(v,\theta,q,S)}{\partial \theta}.
\end{equation*}

We first show that system \eqref{HNS} is strictly hyperbolic for small initial data. Theorem \ref{thm:local} then follows from this strict hyperbolicity by the results of \cite{book-Taylor}.

\begin{lemma}\label{lem:hyperbolic}
There exists a positive constant $\delta,$ such that if $|(v-1,\theta-1,q,S)| \leq \delta$, then the system \eqref{HNS} is strictly hyperbolic.
\end{lemma}
\begin{proof}
From \eqref{p-}, \eqref{e-} and \eqref{gibbs}, the equations \eqref{HNS}$_2$ and  \eqref{HNS}$_3$ are equivalent to 
\begin{equation*}
u_t+p_vv_x+p_\theta\theta_x+p_qq_x+(p_S-1)S_x=0
\end{equation*}
and
\begin{equation*}
e_\theta\theta_t+\theta p_\theta u_x-\frac{2q}{\theta}\theta_x+q_x=\frac{2vq^2}{\kappa\theta}+\frac{vS^2}{\mu}.
\end{equation*}
Denote by $\boldsymbol{u}=(v,u,\theta,q,S)^{\mathrm{T}},$ then the hyperbolic Navier-Stokes equations  \eqref{HNS} can be rewritten as
\begin{equation*}
\boldsymbol{u}_t+\boldsymbol{A}(\boldsymbol{u})\boldsymbol{u}_x=\boldsymbol{g}(\boldsymbol{u}),
\end{equation*}
where
\begin{equation*}
\boldsymbol{A}(\boldsymbol{u})=\begin{pmatrix}
0&-1&0&0&0\\
p_v&0&p_\theta&p_q&p_S-1\\
0&\frac{\theta p_\theta}{e_\theta}&-\frac{2q}{\theta e_\theta}&\frac{1}{e_\theta}&0\\
0&0&\frac{\kappa}{\tau_1 v}&0&0\\
0&-\frac{\mu}{\tau_2 v}&0&0&0
\end{pmatrix},\quad\quad
\boldsymbol{g}(\boldsymbol{u})=\begin{pmatrix}
0\\0\\\frac{2vq^2}{\kappa\theta e_\theta}+\frac{vS^2}{\mu e_\theta}\\-\frac{q}{\tau_1}\\-\frac{S}{\tau_2} 
\end{pmatrix}.
\end{equation*}
It can be calculated that its characteristic polynomial is:
\begin{equation}\label{eigenpolynomial}
\det(\lambda\boldsymbol{E}-\boldsymbol{A})=\lambda \Pi(\lambda),
\end{equation}
where
\begin{equation*}
\Pi(\lambda)=\lambda^4+\frac{2q}{\theta e_\theta}\lambda^3-\left(c_0^2+\frac{\kappa}{\tau_1ve_\theta}+\frac{\theta p_\theta^2}{e_\theta} \right)\lambda^2-\left(\frac{2q}{\theta e_\theta}c_0^2+\frac{\kappa\theta p_\theta p_q}{\tau_1ve_\theta} \right)\lambda+\frac{\kappa}{\tau_1ve_\theta}c_0^2.
\end{equation*}
Here
\begin{equation}\label{c0}
c_0=\sqrt{\frac{\mu(1-p_S)}{\tau_2v}-p_v}.
\end{equation}
Notice that
\begin{equation*}
\Pi(\pm c_0)=-c_0\left(\frac{\theta p_\theta^2}{e_\theta}c_0\pm\frac{\kappa\theta p_\theta p_q}{\tau_1ve_\theta} \right),\quad\quad \Pi(0)=\frac{\kappa}{\tau_1ve_\theta}c_0^2,
\end{equation*}
and that when $|(v-1,\theta-1,q,S)|$ is small enough, it holds that
\begin{equation*}
\frac{\mu(1-p_S)}{\tau_2v}-p_v>0,\quad\quad \frac{\theta p_\theta^2}{e_\theta}c_0\pm\frac{\kappa\theta p_\theta p_q}{\tau_1ve_\theta}>0
\end{equation*}
by \eqref{p-}, \eqref{e-}, \eqref{hypothesis} and the fact that 
$$|p_S|+|p_q|\leq C\delta.$$
Hence the function $\Pi(\lambda)$ has four roots, which means that the characteristic polynomial  $\det(\lambda\boldsymbol{E}-\boldsymbol{A})$ has five different real roots $\lambda_1,\lambda_2,\lambda_3,\lambda_4,\lambda_5,$ satisfying
\begin{equation*}
\lambda_1<-c_0<\lambda_2<\lambda_3=0<\lambda_4<c_0<\lambda_5.
\end{equation*}
This means that the Navier-Stokes equations \eqref{HNS} are indeed hyperbolic.
\end{proof}

\section{The Riemann invariant}
In this section, we give the definition of the Riemann invariant of hyperbolic Navier-Stokes equations \eqref{HNS}.

To begin with, we define the functions $a,b,c,d$ of the variables $v,\theta,q,S:$
\begin{align*}
a(v,\theta,q,S):=& \frac{2q}{\theta e_\theta},\\
b(v,\theta,q,S):=&-\left(c_0^2+\frac{\kappa}{\tau_1ve_\theta}+\frac{\theta p_\theta^2}{e_\theta} \right),\\
c(v,\theta,q,S):=&-\left(\frac{2q}{\theta e_\theta}c_0^2+\frac{\kappa\theta p_\theta p_q}{\tau_1ve_\theta} \right),\\
d(v,\theta,q,S):=&\frac{\kappa}{\tau_1ve_\theta}c_0^2,
\end{align*}
where the pressure $p(v,\theta,q,S)$ and the internal energy $e(v,\theta,q,S)$ is given in \eqref{p} and \eqref{e}. Here $c_0=c_0(v,\theta,q,S)$ is given in \eqref{c0}.

Then the function $\Pi(\lambda)$ in the characteristic polynomial \eqref{eigenpolynomial} can be rewritten as 
\begin{equation}\label{Lambda}
\Pi(\lambda)=\lambda^4+a\lambda^3+b\lambda^2+c\lambda+d.
\end{equation}

Let $c_{0*},a_*,b_*,c_*$ and $d_*$ be the value of the functions $c_{0}(v,\theta,q,S),$ $a(v,\theta,q,S),$ $b(v,\theta,q,S),$ $c(v,\theta,q,S)$ and $d(v,\theta,q,S)$ at equilibrium state $\bar{\boldsymbol{u}}=(1,0,1,0,0)^{\mathrm{T}}.$ Let $\boldsymbol{A}_*$ be the matrix of $\boldsymbol{A}(\boldsymbol{u})$ at equilibrium state $\bar{\boldsymbol{u}}$ and $\lambda_*$ be the value of the second eigenvalue $\lambda_2$ or the fourth eigenvalue $\lambda_4$ at equilibrium state $\bar{\boldsymbol{u}}$. It can be observed that all these constants $a_*,b_*,c_*,d_*,\lambda_*$ and constant matrix $\boldsymbol{A}_*$ are related to the relaxation parameters $\tau_1,\tau_2,$ the viscosity coefficient $\mu,$ the heat conduction $\kappa,$ and the specific heat capacity at constant volume $C_V$ (or the adiabatic index $\gamma$).
Thus, we regard these constants or constant matrix as functions of the relaxation parameters $\tau_1,\tau_2,$ the viscosity coefficient $\mu$, the heat conduction $\kappa,$ and the adiabatic index $\gamma$. In the rest of this article, we denote
\begin{equation}\label{wz}
w=\frac{\kappa}{\tau_1C_{\mathrm{v}}R}>0,\quad\quad\quad z=\frac{\mu}{\tau_2R}>0  
\end{equation}
for brevity.

\subsection{The eigenvector near equilibrium state}
At equilibrium state $\bar{\boldsymbol{u}}=(1,0,1,0,0)^{\mathrm{T}},$ $c_{0*}=\sqrt{R(1+z)}.$ Hence, we have
\begin{equation*}
(a_*,b_*,c_*,d_*)=(0, -R(z+w+\gamma), 0, R^2(z+1)w).
\end{equation*}
This means the eigenvalue $\lambda_*$ at equilibrium state satisfies
\begin{equation*}
\lambda_*^4-R(z+w+\gamma)\lambda_*^2+R^2(z+1)w=0
\end{equation*}
by \eqref{Lambda}. Note that $\lambda_*^2$ satisfies a quadratic equation, whose discriminant is positive:
\begin{equation*}
R^2\left[(z+w+\gamma)^2-4(z+1)w\right]=R^2\left[(z-w+\gamma)^2+4w(\gamma-1)\right]>0.
\end{equation*}
Meanwhile, $\lambda_*$ is the value of the second eigenvalue $\lambda_2$ or the fourth eigenvalue $\lambda_4$ at equilibrium state. Thus, the eigenvalue $\lambda_*$ satisfies
\begin{equation}\label{lambda2}
\lambda_*^2=\frac{(z+w+\gamma)-\sqrt{(z-w+\gamma)^2+4w(\gamma-1)}}{2}R \\
=:\mathcal{L}R,
\end{equation}
where
\begin{equation}\label{L}
\begin{split}
\mathcal{L}=&\frac{(z+w+\gamma)-\sqrt{(z+w+\gamma)^2-4(z+1)w}}{2}  \\
=&\frac{(z+w+\gamma)-\sqrt{(z-w+\gamma)^2+4w(\gamma-1)}}{2}
\end{split}
\end{equation}
is the smaller root of the quadratic equation 
\begin{equation}\label{eqL}
x^2-(z+w+\gamma)x+(z+1)w=0.
\end{equation}

Given the eigenvalue $\lambda_*$ at equilibrium state in hand, now we turn to find the eigenvector at equilibrium state. Notice that
\begin{equation*}
\boldsymbol{A}_*=\begin{pmatrix}
0&-1&0&0&0\\
-R&0&R&0&-1\\
0&\gamma-1&0&\frac{\gamma-1}{R}&0\\
0&0&\frac{R^2}{\gamma-1}w&0&0\\
0&-Rz&0&0&0
\end{pmatrix}. 
\end{equation*}
The eigenvector $\boldsymbol{r}_*$ of the matrix $\boldsymbol{A}_*,$ which satisfies $\boldsymbol{A}_*\boldsymbol{r}_*=\lambda_*\boldsymbol{r}_*$, can be written as
\begin{equation*}
\boldsymbol{r}_*=\left(Rw-\lambda_*^2,\ -\lambda_*(Rw-\lambda_*^2),\ (\gamma-1)\lambda_*^2,\ R^2w\lambda_*,\ Rz(Rw-\lambda_*^2)\right)^{\mathrm{T}}. 
\end{equation*}
Substituting the value of $\lambda_*$ in \eqref{lambda2} into the eigenvector $\boldsymbol{r}_*$, one has
\begin{equation}\label{eigenvector-r}
\boldsymbol{r}_*=\left(R(w-\mathcal{L}),\ -\lambda_*R(w-\mathcal{L}),\ (\gamma-1)R\mathcal{L},\ R^2w\lambda_*,\ R^2z(w-\mathcal{L})\right)^{\mathrm{T}}.
\end{equation}

\subsection{The gradient of eigenvalue near equilibrium state}
Since the Riemann invariant is the dot product of the gradient of the eigenvalue $\lambda$ and the right eigenvector $\boldsymbol{r}_*,$ we need to find the gradient of the eigenvalue near equilibrium state.

We write the expression of function $\Pi(\lambda)$ in \eqref{Lambda} as
\begin{equation*}
\Pi(v,\theta,q,S,\lambda)=\lambda^4+a(v,\theta,q,S)\lambda^3+b(v,\theta,q,S)\lambda^2+c(v,\theta,q,S)\lambda+d(v,\theta,q,S).
\end{equation*}
The fact that the eigenvalue $\lambda(v,\theta,q,S)$ solves
\begin{equation*}
\Pi(v,\theta,q,S,\lambda(v,\theta,q,S))=0
\end{equation*}
implies
\begin{equation}\label{parellel}
\nabla_{(v,\theta,q,S)}\lambda(v,\theta,q,S) \quad\pll\quad \nabla_{(v,\theta,q,S)}\Pi(v,\theta,q,S,\lambda)
\end{equation}
by chain rule. This means we only need to find the value of $\nabla_{(v,\theta,q,S)}\Pi(v,\theta,q,S,\lambda)$ near equilibrium state.

Note that all the first-order partial derivatives and the second-order partial derivatives of the pressure $p(v,\theta,q,S)$ and the internal energy $e(v,\theta,q,S)$ near equilibrium state $\bar{\boldsymbol{u}}=(1,0,1,0,0)^{\mathrm{T}}$ are zero, except that
\begin{align*}
&\frac{\partial p}{\partial v}=\frac{\partial^2 p}{\partial v\partial\theta}=-R,&& \frac{\partial p}{\partial \theta}=R, &&\frac{\partial^2 p}{\partial v^2}=2R,&&\frac{\partial e}{\partial \theta}=C_{\mathrm{v}},\\
&\frac{\partial^2 p}{\partial S^2}=-\frac{\tau_2}{\mu},&& \frac{\partial^2 p}{\partial q^2}=-\frac{\tau_1}{\kappa},&& \frac{\partial^2 e}{\partial S^2}=\frac{\tau_2}{\mu},&& \frac{\partial^2 e}{\partial q^2}=\frac{2\tau_1}{\kappa}.
\end{align*}
From the above observations we can obtain
\begin{align*}
\nabla_{(v,\theta,q,S)}\frac{2q}{\theta e_\theta}=&\left(0,0,\frac{2}{C_{\mathrm{v}}},0\right),  \\
\nabla_{(v,\theta,q,S)}c_0^2=&\left(-\frac{\mu}{\tau_2}-2R,R,0,1\right),\\
\nabla_{(v,\theta,q,S)}\frac{\kappa}{\tau_1 ve_\theta}=&\left(-\frac{\kappa}{\tau_1 C_{\mathrm{v}}},0,0,0\right),\\
\nabla_{(v,\theta,q,S)}\frac{\theta p_\theta^2}{e_\theta}=&\left(-\frac{2R^2}{C_{\mathrm{v}}},\frac{R^2}{C_{\mathrm{v}}},0,0\right),\\
\nabla_{(v,\theta,q,S)}\frac{\kappa\theta p_\theta p_q}{\tau_1v e_\theta}=&\left(0,0,-\frac{R}{C_{\mathrm{v}}},0\right)
\end{align*}
near equilibrium state. Substituting the above results into the expressions of $a(v,\theta,q,S),$ $b(v,\theta,q,S),$ $c(v,\theta,q,S),$ $d(v,\theta,q,S)$ and utilizing the new variants $w, z$ in \eqref{wz} , it holds that
\begin{subequations}\label{nabla-abcd}
\begin{align}
\nabla_{(v,\theta,q,S)}a=&\left(0,0,\frac{2}{C_{\mathrm{v}}},0\right), \label{nabla-a}\\
\nabla_{(v,\theta,q,S)}b=&R\left(z+w+2\gamma,-\gamma,0,-\frac{1}{R} \right), \label{nabla-b}\\
\nabla_{(v,\theta,q,S)}c=&\left(0,0,-(\gamma-1)(2z+1),0\right), \label{nabla-c}\\
\nabla_{(v,\theta,q,S)}d=&wR^2\left(-2z-3,1,0,\frac{1}{R}  \right) \label{nabla-d}
\end{align}
\end{subequations}
near equilibrium state. Since 
\begin{equation*}
\nabla_{(v,\theta,q,S)}\Pi(v,\theta,q,S,\lambda_*)=\lambda_*^3 \nabla_{(v,\theta,q,S)}a+\lambda_*^2\nabla_{(v,\theta,q,S)}b+\lambda_*\nabla_{(v,\theta,q,S)}c+\nabla_{(v,\theta,q,S)}d,
\end{equation*}
we have
\begin{subequations}\label{nabla-L}
\begin{align}
\frac{\partial}{\partial v}\Pi(v,\theta,q,S,\lambda_*)=&R^2\left[(z+w+2\gamma)\mathcal{L}-(2z+3)w\right] ,  \label{nabla-Lv}\\
\frac{\partial}{\partial \theta}\Pi(v,\theta,q,S,\lambda_*)=& R^2(w-\gamma \mathcal{L}) , \label{nabla-Ltheta}\\
\frac{\partial}{\partial q}\Pi(v,\theta,q,S,\lambda_*)=& (\gamma-1)(2\mathcal{L}-2z-1)\lambda_* , \label{nabla-Lq}\\
\frac{\partial}{\partial S}\Pi(v,\theta,q,S,\lambda_*)=&R(w-\mathcal{L})   \label{nabla-LS}
\end{align}
\end{subequations}
near equilibrium state by \eqref{lambda2} and \eqref{nabla-abcd}.

\subsection{The construction of the Riemann invariant}

Since \eqref{parellel} means that the gradient of $\lambda$ and the gradient of $\Pi$ are parallel, the Riemann invariant of the hyperbolic Navier-Stokes equations \eqref{HNS} can be defined as
\begin{equation}\label{def-R}
\mathcal{R}(w,z):=\frac{1}{R^3}\nabla_{(v,u,\theta,q,S)}\Pi_*\cdot\boldsymbol{r}_*,
\end{equation}
where $\Pi_*=\Pi(1,1,0,0,\lambda_*).$

The following lemma gives the expression of the Riemann invariant $\mathcal{R}$ with respect to $z$ and $w$.
\begin{lemma}\label{lem:def-R}
The Riemann invariant $\mathcal{R}$ constructed above satisfies
\begin{equation}\label{R}
\mathcal{R}(w,z)=\mathcal{M}(w,z)\mathcal{L}(w,z)-w\mathcal{N}(w,z),
\end{equation}
where
\begin{align}
\mathcal{M}(w,z)=&2(\gamma-1)w^2+(\gamma^2-2\gamma+3)w-\gamma(\gamma+1)(z+\gamma), \label{M} \\
\mathcal{N}(w,z)=&2\left((\gamma-1)z+\gamma\right)w-\gamma(\gamma+1)(z+1). \label{N}
\end{align}
\end{lemma}

\begin{proof}
Substituting \eqref{eigenvector-r} and \eqref{nabla-L} into \eqref{def-R}, the Riemann invariant $\mathcal{R}$ can be rewritten as
\begin{align*}
\mathcal{R}(w,z)=&\left[(z+w+2\gamma)\mathcal{L}-(2z+3)w\right](w-\mathcal{L})+(w-\gamma \mathcal{L})(\gamma-1)\mathcal{L} \\
&+(\gamma-1)(2\mathcal{L}-2z-1)w\mathcal{L}+z(w-\mathcal{L})^2\\
=&\left[(2\gamma-3)w-\gamma(\gamma+1)\right]\cdot\mathcal{L}^2+\left[(3-2\gamma)z+w+2\gamma+3\right]\cdot w\mathcal{L}\\
&-(z+3)\cdot w^2.
\end{align*}
Note that the function $\mathcal{L}$ satisfies the quadratic equation \eqref{eqL}, we have
\begin{equation*}
\mathcal{L}^2=(z+w+\gamma)\mathcal{L}-(z+1)w.
\end{equation*}
The proof of the lemma is done by substituting the above equality into the expression of $\mathcal{R}(w,z).$
\end{proof}

\section{Genuine nonlinear}

In this section, we prove that the Riemann invariant $\mathcal{R}(w,z)$ is negative, which implies that the hyperbolic Navier-Stokes equations \eqref{HNS} possesses two genuinely nonlinear eigenvalues. Besides, we have proved that the hyperbolic Navier-Stokes functions \eqref{HNS} in the conservation laws form 
\begin{equation*}
\boldsymbol{u}_t+\boldsymbol{A}(\boldsymbol{u})\boldsymbol{u}_x={ \boldsymbol{g}(\boldsymbol{u})}
\end{equation*}
is strictly hyperbolic near constant equilibrium state and $\boldsymbol{g}(\boldsymbol{u})=\boldsymbol{0}.$ Hence, by the work of B\"{a}rlin in \cite{Barlin}, we obtain that the derivatives of the solutions to hyperbolic Navier-Stokes equations \eqref{HNS} blows up in finite time.

\begin{proposition}\label{prop:main}
Suppose that the adiabatic index $\gamma$ satisfies $1<\gamma\leq\frac53.$ Then the Riemann invariant $\mathcal{R}(w,z)$ satisfies
\begin{equation*}
\mathcal{R}(w,z)<0,\quad\quad\quad \forall z>0,\ \forall w>0,
\end{equation*}
which implies the second eigenvalue $\lambda_2$ and the fourth $\lambda_4$ are genuine nonlinear near equilibrium state $\bar{\boldsymbol{u} }=(1,0,1,0,0)^{\mathrm{T}}.$
\end{proposition}

Regarding $\mathcal{M}(w,z)$ as a function of $w$ for fixed $z > 0$, it is quadratic. Its associated equation has two distinct real roots: one positive and one negative. Since $w$ is positive, we ignore the negative one and denote the positive one by $w_M(z).$ Meanwhile, we  denote the root of $\mathcal{N}(w,z)$ as a function of $w$ for fixed $z > 0$ as $w_N(z).$ That is to say
\begin{equation*}
\mathcal{M}(w_M(z),z)=0,\quad\quad\quad\mathcal{N}(w_N(z),z)=0,\quad\quad\quad \forall z>0.
\end{equation*}

The calculation of the Riemann invariant $\mathcal{R}(w,z)$ (defined by \eqref{def-R}) is very complicated because the signs of $\mathcal{M}(w,z)$ and $\mathcal{N}(w,z)$ change. Hence, we compare the magnitudes of the roots $w_{M}(z)$ and $w_N(z)$ as a function of $z.$
\begin{lemma}\label{lem:wM-wN}
For all $1<\gamma\leq\frac53,$ the roots $w_{M}(z)$ and $w_N(z)$ satisfy
\begin{equation}\label{wM-wN}
0<w_N(z)<w_M(z),\quad\quad\quad \forall z>0.
\end{equation}
\end{lemma}
\begin{figure}[htpb]
\centering
\includegraphics[width=0.8\textwidth]{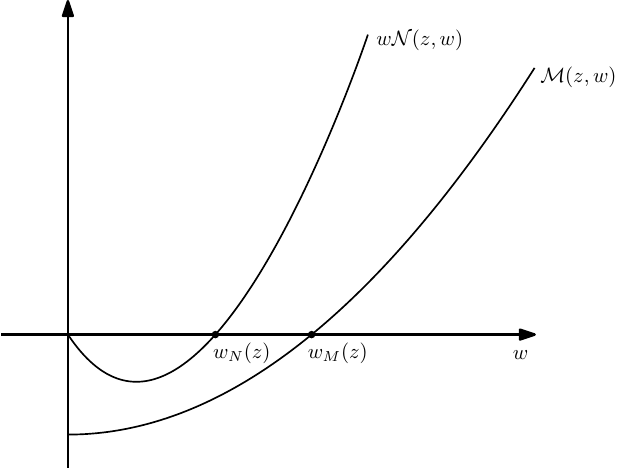}
\caption{The roots of $\mathcal{M}$ and $\mathcal{N}$}
\end{figure}
 
\begin{proof}
To begin with, let
\begin{equation*}
\tilde{z}=(\gamma-1)z+\gamma,
\end{equation*}
then we have $\tilde{z}>\gamma>1.$ The functions $\mathcal{M}$ and $\mathcal{N}$ can be rewritten as
\begin{align*}
\mathcal{M}(w,z)=&2(\gamma-1)w^2+(\gamma^2-2\gamma+3)w-\frac{\gamma(\gamma+1)}{\gamma-1}\left(\tilde{z}+\gamma(\gamma-2)\right), \\
\mathcal{N}(w,z)=&2\tilde{z}w-\frac{\gamma(\gamma+1)}{\gamma-1}\left(\tilde{z}-1\right).
\end{align*}
The roots $w_{M}(z)$ and $w_N(z)$ satisfy
\begin{align}
w_M(z)=&\frac{\sqrt{D_1(\tilde{z})}-(\gamma^2-2\gamma+3)}{4(\gamma-1)}, \label{wM}\\
w_N(z)=&\frac{\gamma(\gamma+1)}{2(\gamma-1)}\cdot \frac{\tilde{z}-1}{\tilde{z} },\label{wN}
\end{align}
where 
\begin{align*}
D_1(\tilde{z})=&(\gamma^2-2\gamma+3)^2+8\gamma(\gamma+1)\left(\tilde{z}+\gamma(\gamma-2)\right)  \\
=&(\gamma^2-2\gamma+3)^2+8\gamma(\gamma+1)\left((\gamma-1)z+\gamma(\gamma-1)\right)>0
\end{align*}
 is the discriminant of $\mathcal{M}(w,z)$ considered as a quadratic function in $w$.

Next, define
\begin{align*}
h_{\pm}(\tilde{z})=&\tilde{z}\sqrt{D_1(\tilde{z})}\pm\left[(\gamma^2-2\gamma+3)\tilde{z}+2\gamma(\gamma+1)\left(\tilde{z}-1\right)\right], \\
h(\tilde{z})=&2\tilde{z}^3-3(\gamma+1)\tilde{z}^2+3(\gamma^2+1)\tilde{z}-\gamma(\gamma+1).
\end{align*}
Notice that 
\begin{equation}\label{te-1}
w_M(z)-w_N(z)=\frac{1}{4(\gamma-1)\tilde{z} }h_-(\tilde{z})
\end{equation}
by \eqref{wM} and \eqref{wN}, and meanwhile
\begin{align*}
h_-(\tilde{z})h_+(\tilde{z})=&\tilde{z}^2D_1(\tilde{z})-\left[(\gamma^2-2\gamma+3)\tilde{z}+2\gamma(\gamma+1)\left(\tilde{z}-1\right)\right]^2\\
=&4\gamma(\gamma+1)h(\tilde{z}).
\end{align*}
Multiplying \eqref{te-1} by $(\gamma-1)\tilde{z} h_+(\tilde{z}),$ we have
\begin{equation}\label{eq-im}
(\gamma-1)\tilde{z}h_+(\tilde{z})\left(w_M(z)-w_N(z)\right)=\gamma(\gamma+1)h(\tilde{z}).
\end{equation}

Lastly, we aim to prove $w_M(z)$ is greater than $w_N(z).$ To this end, by the fact that
\begin{equation*}
\gamma>1,\quad\quad\quad\tilde{z}>1,\quad\quad\quad h_+(\tilde{z})>0, 
\end{equation*}
it suffices to show that $h(\tilde{z})>0$ for all $z>0.$ Notice that
\begin{equation*}
h(\gamma)=2\gamma(\gamma-1)^2>0,
\end{equation*}
and that 
\begin{align*}
h^\prime(\tilde{z})=&6\tilde{z}^2-6(\gamma+1)\tilde{z}+3(\gamma^2+1)\\ 
=&6\left[\left(\tilde{z}-\frac{\gamma+1}{2}  \right)^2+\left(\frac{\gamma-1}{2} \right)^2\right]\\
>&0.
\end{align*}
Hence, we have
\begin{equation*}
h(\tilde{z})>h(\gamma)>0
\end{equation*}
by the fact that $\tilde{z}=(\gamma-1)z+\gamma>\gamma.$ Therefor, the lemma holds by \eqref{eq-im}.
\end{proof}
 
We now find two new expressions for the Riemann invariant $\mathcal{R}(w,z).$ To do this, we split $\mathcal{L}(w,z)$ defined in \eqref{L} into two parts:
\begin{equation*}
\mathcal{L}(w,z)=\frac12 f(w,z)-\frac12 g(w,z),
\end{equation*}
where
\begin{equation}\label{g}
f(w,z)=z+w+\gamma,\quad\quad\quad g(w,z)=\sqrt{(z+w+\gamma)^2-4(z+1)w}.
\end{equation}
Meanwhile, let
\begin{align*}
\mathcal{Q}(w,z):=&\mathcal{M}(w,z)f(w,z)-2w\mathcal{N}(w,z), \\
\mathcal{P}(w,z):=&\frac{1}{4(\gamma-1)w^2}\left(\mathcal{M}(w,z)^2g(w,z)^2-\mathcal{Q}(w,z)^2\right).
\end{align*}
Without ambiguity, we denote briefly the functions $f(w,z),$ $g(w,z),$ $\mathcal{M}(w,z),$ $\mathcal{N}(w,z),$ $\mathcal{R}(w,z),$ $\mathcal{Q}(w,z),$ $\mathcal{P}(w,z)$ as $f,$ $g,$ $\mathcal{M},$ $\mathcal{N},$ $\mathcal{R},$ $\mathcal{Q},$ $\mathcal{P}.$ Then it holds that
\begin{equation}\label{relation1}
\mathcal{R}=\frac12\left(\mathcal{Q}-\mathcal{M}g\right)
\end{equation}
and 
\begin{equation}\label{relation2}
2(\gamma-1)w^2\mathcal{P}=-\left(\mathcal{M}g+\mathcal{Q}\right)\mathcal{R}.
\end{equation}
 
The following two lemmas study the properties of the functions $\mathcal{P}(w,z)$ and $\mathcal{Q}(w,z).$ 

\begin{lemma}\label{lem:P}
Suppose that the adiabatic index $\gamma$ satisfies $1<\gamma\leq\frac53.$ Fixing any $z>0,$ we obtain that the function $\mathcal{P}(w,z)$ is a cubic function in the variable $w$:
\begin{equation}\label{P-}
\mathcal{P}(w,z)=4w^3-p_2(z)w^2+p_1(z)w-p_0(z),
\end{equation}
where the coefficients $p_i(z)\ (i=0,1,2)$ are positive, and
\begin{align*}
p_2(z)=&6(\gamma+1)z+12, \\
p_1(z)=&6(\gamma^2+1)z^2+3(3\gamma^2+2\gamma+3)z+3(\gamma^2+3),\\
p_0(z)=&\gamma(\gamma+1)\left[2z^3+3(\gamma+1)z^2+3(\gamma+1)z+2\gamma\right].
\end{align*}
Besides, the function $\mathcal{P}(w,z)$ is monotonically increasing from negative to positive infinity with respect to the independent variable $w,$ i.e.,
\begin{equation*}
\frac{\partial}{\partial w}\mathcal{P}(w,z)>0,\quad\quad\quad \forall w>0,\quad \forall z>0,
\end{equation*}
and the function $\mathcal{P}(w,z),$ as a function of $w$, has only one root $w_P(z),$ i.e.,
\begin{equation*}
\mathcal{P}(w_P(z),z)=0,\quad\quad\quad \forall z>0.
\end{equation*}
\end{lemma}
 
\begin{proof}
To begin with, we deduce the expression of $\mathcal{P}.$ By the definition of $\mathcal{P}(w,z)$ and $\mathcal{Q}(w,z)$, one has
\begin{align}
(\gamma-1)\mathcal{P}=&\frac{1}{4w^2}\left[\mathcal{M}^2g^2-\left(\mathcal{M}f-2w\mathcal{N}\right)^2\right] \nonumber\\
=&\frac{1}{w^2}\left[-\frac14(f^2-g^2)\mathcal{M}+wf\mathcal{N}\right]\mathcal{M}-\mathcal{N}^2.\label{te-2}
\end{align}
Notice that 
\begin{equation*}
f^2-g^2=4(z+1)w
\end{equation*}
by \eqref{g} and 
\begin{equation*}
f\mathcal{N}-(z+1)\mathcal{M}=w\left[2w+2(\gamma-1)z^2-(3-\gamma)(z+1)\right].
\end{equation*}
Substituting the above two equations into \eqref{te-2}, one has 
\begin{align*}
(\gamma-1)\mathcal{P}=&\frac{1}{w} \left[f\mathcal{N}-(z+1)\mathcal{M}\right]\mathcal{M}-\mathcal{N}^2 \\
=&\left[2w+2(\gamma-1)z^2-(3-\gamma)(z+1)\right]\mathcal{M}-\mathcal{N}^2\\
=&(\gamma-1)\left(4w^3-p_2(z)w^2+p_1(z)w-p_0(z)\right)
\end{align*}
by the definition of $\mathcal{M}$ and $\mathcal{N}$ in Lemma \ref{lem:def-R}. This means \eqref{P-} holds.

Next, we prove the monotonicity of $\mathcal{P}(w,\cdot).$ A straightforward calculation shows that
\begin{equation*}
\frac{\partial}{\partial w}\mathcal{P}(w,z)=12w^2-2p_2(z)w+p_1(z),
\end{equation*}
whose discriminant $D_2(z)$ satisfies
\begin{align*}
D_2(z)=&4p_2(z)^2-48p_1(z)\\
=&-144\left[(\gamma-1)^2z^2+(\gamma-1)(3\gamma+1)z+(\gamma^2-1)\right] \\
<&0
\end{align*}
for all $z>0$ and $\gamma\in(1,\frac53].$ Hence, 
\begin{equation*}
\frac{\partial}{\partial w}\mathcal{P}(w,z)=12\left(w-\frac1{12}p_2(z) \right)^2-\frac1{48}D_2(z)>0.
\end{equation*}
This means that the function $\mathcal{P}(w,z)$ is monotonically increasing with respect to the independent variable $w.$

Finally, we prove that the function $\mathcal{P}(w,z),$ as a function of $w$, has only one root. Notice that 
\begin{equation*}
\mathcal{P}(0,z)=-p_0(z)<0,\quad\quad\quad \lim\limits_{w\to\infty}\mathcal{P}(w,z)=+\infty.
\end{equation*}
By the Intermediate Value Theorem and monotonicity of $\mathcal{P}(w,z),$ the function $\mathcal{P}(w,z),$ as a function of $w$, has exactly one zero $w_P(z).$
\end{proof}

\begin{lemma}\label{lem:Q}
Suppose that the adiabatic index $\gamma$ satisfies $1<\gamma\leq\frac53.$ Fixing any $z>0,$ we obtain that the function $\mathcal{Q}(w,z)$ is a cubic function in the variable $w$:
\begin{equation}\label{Q-}
\mathcal{Q}(w,z)=2(\gamma-1)w^3-q_2(z)w^2+q_1(z)w-q_0(z),
\end{equation}
where the coefficients $q_i(z)\ (i=0,1,2)$ are positive, and
\begin{align*}
q_2(z)=&2(\gamma-1)z+\left(8\gamma-3\gamma^2-3\right), \\
q_1(z)=&(2\gamma^2-\gamma+3)z+\gamma(5-\gamma),\\
q_0(z)=&\gamma(\gamma+1)(z+\gamma)^2.
\end{align*}
Besides, the function $\mathcal{Q}(w,z),$ as a function of $w$, has one to three roots. We denote the first(smallest) root of $\mathcal{Q}(w,z)$ as $w_Q(z),$ which means 
\begin{equation}\label{zero-wQ}
\mathcal{Q}(w,z)<0,\quad\quad\quad \forall w\in(0,w_Q(z)),\quad \forall z>0.
\end{equation}
\end{lemma}
 
\begin{proof}
Substituting the definition of $f(w,z),$ $\mathcal{M}(w,z)$ and $\mathcal{N}(w,z)$ into the definition of $\mathcal{Q},$ one has \eqref{Q-}. Notice that 
\begin{equation*}
\mathcal{Q}(0,z)=-q_0(z)<0,\quad\quad\quad \lim\limits_{w\to\infty}\mathcal{Q}(w,z)=+\infty.
\end{equation*}
By the Intermediate Value Theorem and the fact that $\mathcal{Q}(w,z)$ is a cubic function in the variable $w,$ the function $\mathcal{Q}(w,z),$ as a function of $w$, must have one to three zeros and \eqref{zero-wQ} holds.
\end{proof}
 
Actually, the function $\mathcal{Q}(w,z)$ must satisfy one of the following two conclusions:
\begin{enumerate}
\item [(a)] the function $\mathcal{Q}(w,z)$ is monotonically increasing with respect to the independent variable $w$;
\item [(b)] the function $\mathcal{Q}(w,z)$ first increases, then decreases, and finally increases again with respect to the independent variable $w.$
\end{enumerate}
In Lemma \ref{lem:mono-Q} in the appendix, we provide a necessary and sufficient condition for (a) to hold.

Then we give the relationship of the roots $w_N(z),$ $w_M(z),$ $w_P(z),$ $w_Q(z).$
\begin{lemma}\label{lem:roots}
Suppose that the adiabatic index $\gamma$ satisfies $1<\gamma\leq\frac53.$ Then the roots $w_N(z),$ $w_M(z),$ $w_P(z),$ $w_Q(z)$ of functions  $\mathcal{N}(w,z),$ $\mathcal{M}(w,z),$  $\mathcal{P}(w,z)$ $\mathcal{Q}(w,z)$ satisfy
\begin{equation}\label{relation}
0<w_N(z)<w_M(z)<w_P(z)\leq w_Q(z),\quad\quad\quad \forall z>0.
\end{equation}
\end{lemma}
\begin{proof}
The proof relies mainly on the monotonicity of the function $\mathcal{P}(w,\cdot).$ Notice that 
\begin{align*}
\mathcal{P}(w_M(z),z)=&\frac{1}{4(\gamma-1)w_M(z)^2}\left(\mathcal{M}(w_M(z),z)^2g(w_M(z),z)^2-\mathcal{Q}(w_M(z),z)^2\right)\\
=&-\frac{1}{4(\gamma-1)w_M(z)^2}\mathcal{Q}(w_M(z),z)^2
\end{align*}
by the fact that $\mathcal{M}(w_M(z),z)=0.$ Since the Lemma \ref{lem:wM-wN} and the monotonicity of $\mathcal{N}(w,\cdot)$ imply
\begin{equation*}
\mathcal{N}(w_M(z),z)>\mathcal{N}(w_N(z),z)=0,
\end{equation*}
we have 
\begin{align*}
\mathcal{Q}(w_M(z),z)=&\mathcal{M}(w_M(z),z)f(w_M(z),z)-2w_M(z)\mathcal{N}(w_M(z),z)\\
=&-2w_M(z)\mathcal{N}(w_M(z),z)<0.
\end{align*}
These observations mean that
\begin{equation}\label{PwM}
\mathcal{P}(w_M(z),z)<0,\quad\quad\quad\forall z>0.
\end{equation}

Next, note that
\begin{align*}
\mathcal{P}(w_Q(z),z)=&\frac{1}{4(\gamma-1)w_Q(z)^2}\left(\mathcal{M}(w_Q(z),z)^2g(w_Q(z),z)^2-\mathcal{Q}(w_Q(z),z)^2\right)\\
=&\frac{1}{4(\gamma-1)w_Q(z)^2}\mathcal{M}(w_Q(z),z)^2g(w_Q(z),z)^2
\end{align*}
by the fact that $\mathcal{Q}(w_Q(z),z)=0.$ We obtain
\begin{equation}\label{PwQ}
\mathcal{P}(w_Q(z),z)\geq0,\quad\quad\quad\forall z>0.
\end{equation}

Since the function $\mathcal{P}(w,\cdot )$ is monotonically increasing and 
\begin{equation*}
\mathcal{P}(w_M(z),z)<\mathcal{P}(w_P(z),z)\leq \mathcal{P}(w_Q(z),z)
\end{equation*}
by \eqref{PwM}, \eqref{PwQ} and $\mathcal{P}(w_P(z),z)=0,$  one has
\begin{equation*}
w_M(z)<w_P(z)\leq w_Q(z),\quad\quad\quad\forall z>0.
\end{equation*}
This combined with \eqref{wM-wN} finishes the proof.
\end{proof}

\begin{proof}[The proof of Proposition \ref{prop:main}]
Fixing any $z>0.$ From Lemma \ref{lem:roots}, we can partition the interval  $w\in (0,+\infty)$ into the union of four disjoint subsets:
\begin{equation*}
(0,+\infty)=(0,w_N(z)]\cup (w_N(z),w_M(z)]\cup \mathcal{I}_-(z)\cup\mathcal{I}_+(z),
\end{equation*}
where
\begin{align*}
\mathcal{I}_-(z)=&\{w>w_M(z):\mathcal{Q}(w,z)\leq0\}, \\
\mathcal{I}_+(z)=&\{w>w_M(z):\mathcal{Q}(w,z)>0\}.
\end{align*}

Now, fixing any $z>0,$ we prove $\mathcal{R}(w,z)<0$ for all $w>0$ in these four subsets.
\begin{enumerate}
\item [(a)] For all $w\in (0,w_N(z)],$ we have 
$$\mathcal{M}(w,z)<0,\quad\quad\quad \mathcal{N}(w,z)\leq0$$
by \eqref{wM-wN}. Since $(0,w_N(z)]\subset (0,w_Q(z))$ by Lemma \ref{lem:roots}, it holds that 
\begin{equation*}
\mathcal{Q}(w,z)<0,\quad\quad\quad \forall w\in (0,w_N(z)).
\end{equation*}
As for the function $\mathcal{P}(w,z),$ we have
\begin{equation*}
\mathcal{P}(w,z)<\mathcal{P}(w_P(z),z)=0,\quad\quad\quad \forall w\in (0,w_N(z))
\end{equation*}
by Lemma \ref{lem:P} and Lemma \ref{lem:roots}. Then, by the fact that $\gamma>1$ and $g(w,z)>0,$ one has $\mathcal{M}g+\mathcal{Q}\neq0.$ Hence, it holds that
\begin{equation*}
\mathcal{R}=-\frac{2(\gamma-1)w^2\mathcal{P}}{\mathcal{M}g+\mathcal{Q}}<0
\end{equation*}
by \eqref{relation2}.
\item [(b)] For all $w\in (w_N(z),w_M(z)],$ we have 
$$\mathcal{M}(w,z)\leq0,\quad\quad\quad \mathcal{N}(w,z)>0$$
by Lemma \ref{wM-wN}. Hence, it holds that
\begin{equation*}
\mathcal{R}=\mathcal{M}\mathcal{L}-w\mathcal{N}<0
\end{equation*}
by the fact that $\mathcal{L}(w,z)>0$ for all $w>0,$ $z>0.$
\item [(c)] For all $w\in\mathcal{I}_-(z),$ we have
\begin{equation*}
\mathcal{M}(w,z)>0,\quad\quad\quad\mathcal{Q}(w,z)\leq0
\end{equation*}
by the definition of domain $\mathcal{I}_-(z)$ and Lemma \ref{lem:wM-wN}. Then, by the fact that $g(w,z)>0$ and the relation \eqref{relation1}, we have
\begin{equation*}
\mathcal{R}=\frac12(\mathcal{Q}-\mathcal{M}g)<0.
\end{equation*}
\item [(d)] For all $w\in\mathcal{I}_+(z),$ we have
\begin{equation*}
\mathcal{M}(w,z)>0,\quad\quad\quad\mathcal{Q}(w,z)>0
\end{equation*}
by the definition of domain $\mathcal{I}_+(z)$ and Lemma \ref{lem:wM-wN}. Since $w_Q(z)$ is the first root of $\mathcal{Q}(w,z),$ the inequality $\mathcal{Q}(w,z)>0$ implies $w>w_Q(z)$ for any fixed $z>0.$ Then we have
\begin{equation*}
\mathcal{P}(w,z)>\mathcal{P}(w_Q(z),z)\geq\mathcal{P}(w_P(z),z)=0
\end{equation*}
by the fact that $\mathcal{P}(w,\cdot)$ is monotone increasing and $w>w_Q(z)\geq w_P(z)$ from Lemma \ref{lem:P} and Lemma \ref{lem:roots}. Thus, we have
\begin{equation*}
\mathcal{R}=-\frac{2(\gamma-1)w^2\mathcal{P}}{\mathcal{M}g+\mathcal{Q}}<0
\end{equation*}
by \eqref{relation2}.
\end{enumerate}

\end{proof}

\appendix

\section{Property of $\mathcal{Q}(w,z)$}
In this appendix, wo give the necessary and sufficient conditions for the function $\mathcal{Q}(w,z)$ to be monotonically increasing with respect to $w$.

\begin{lemma}\label{lem:mono-Q}
Suppose that the adiabatic index $\gamma$ satisfies $1<\gamma\leq\frac53.$ Then the function $\mathcal{Q}(w,z)$ is monotonically increasing for $w\in (0,+\infty)$ if and only if one of the following conditions holds:
\begin{enumerate}
\item [(a)] $\gamma\in (1,\gamma_0)$ and $z\in [z_-,z_+];$
\item [(b)] $\gamma\in [\gamma_0,\frac53]$ and $z\in (0,z_+],$
\end{enumerate}
where $\gamma_0$ is the unique zero of the function 
\begin{equation*}
a_0(\gamma)=9\gamma^4-42\gamma^3+46\gamma^2-18\gamma+9=0
\end{equation*}
 on the interval $(1,\frac53],$ and
\begin{equation*}
z_\pm=\frac{12\gamma^2-19\gamma+15\pm \sqrt{3(36\gamma^4-96\gamma^3+179\gamma^2-166\gamma+63)}}{4(\gamma-1)}.
\end{equation*}
\end{lemma}
\begin{proof}
By the expression of $\mathcal{Q}(w,z)$ in Lemma \ref{lem:Q}, we have
\begin{equation*}
\frac{\partial}{\partial w}\mathcal{Q}(w,z)=6(\gamma-1)w^2-2q_2(z)w+q_1(z).
\end{equation*}
The discriminant of above quadratic equation is 
\begin{align*}
D_2(z)=&4q_2(z)^2-24(\gamma-1)q_1(z) \\
=&4\left[4(\gamma-1)^2z^2-2(\gamma-1)a_1(\gamma)z+a_0(\gamma)\right],
\end{align*}
where
\begin{equation*}
a_1(\gamma)=12\gamma^2-19\gamma+15.
\end{equation*}

Now we prove that the discriminant function $D_2(z)$ has two roots for $z\in \mathbb{R}.$ Since $D_2(z)$ is still a quadratic equation in $z,$ we denote its discriminant as 
\begin{align*}
D_3=&\left[2(\gamma-1)a_1(\gamma)\right]^2-16(\gamma-1)^2a_0(\gamma)\\
=&12(\gamma-1)^2(36\gamma^4-96\gamma^3+179\gamma^2-166\gamma+63)\\
>&0
\end{align*}
for $1<\gamma\leq\frac53$. This implies $D_2(z)$ has two roots $z_\pm$ in $\mathbb{R}.$ However, the domain we discuss is the half space $(0,+\infty).$ Notice that the function $a_0(\gamma)$ is monotone decreasing  and changes sign from positive to negative when $1<\gamma\leq \frac53$ (see Figure \ref{fig:a0}).  We have
\begin{align*}
z_-+z_+=&\frac{a_1(\gamma)}{2(\gamma-1)}>0, \quad\quad\quad\quad \ \gamma\in \left(1,\frac53\right],\\
z_-z_+=&\frac{a_0(\gamma)}{4(\gamma-1)^2}
 \left\{\begin{aligned}
&>0,\quad\quad\quad \gamma\in(1,\gamma_0),  \\
&\leq 0, \quad\quad\quad \gamma\in \left[\gamma_0,\frac53\right],
\end{aligned}\right.
\end{align*}
which means 
\begin{equation*}
\frac{\partial}{\partial w}\mathcal{Q}(w,z)>0,\quad \forall w>0 \quad\quad\quad \Longleftrightarrow \quad\quad\quad z\in [z_-,z_+]\cap(0,+\infty).
\end{equation*}
Hence, when $\gamma\in(1,\gamma_0),$ we have $z_\pm>0,$ and then (a) holds. When $\gamma\in \left[\gamma_0,\frac53\right],$ we have $z_-\leq0<z_+.$ Therefore, we replace interval $(0,z_+]$ with interval $[z_-,z_+]$.
\begin{figure}[htpb]
\centering
\includegraphics[width=\columnwidth]{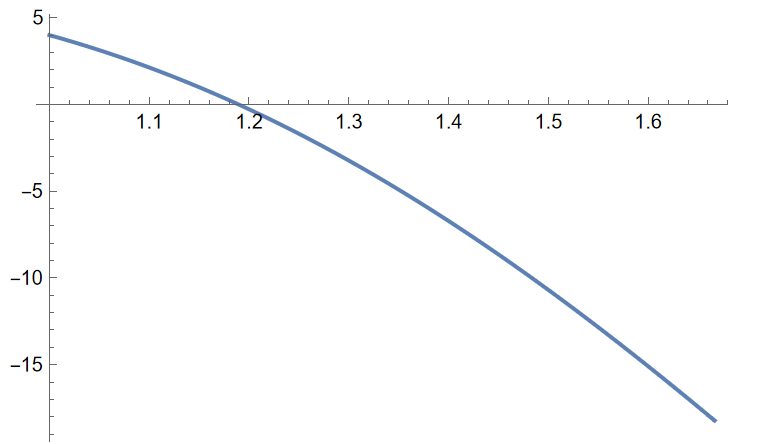} 
\caption{The graph of $a_0(\gamma)$} \label{fig:a0}
\end{figure}
 
\end{proof}

\section{Equation of state for hyperbolic Navier-Stokes equations}

In this appendix, we explain why we should add the correction terms of $q, S$ into the expressions of the pressure $p$ in \eqref{p-}, the internal energy $e$ in \eqref{e-},  and the entropy $s$ in \eqref{s} for hyperbolic Navier-Stokes equations.

%We consider the model with viscosity and heat conduction satisfying Cattaneo's law:
%\begin{equation}\label{NS-CL}
% \left\{\begin{aligned}
%&v_t-u_x=0,  \\
%&u_t+p_x=S_x,\\
%&e_t+pu_x+q_x=Su_x,\\
%&\tau_1q_t+q+\frac{\kappa\theta_x}{v}=0,\\
%&\tau_2S_t+S=\frac{\mu u_x}{v}. 
%\end{aligned}\right.
%\end{equation}

%First we observe that
%\begin{gather*}
%\left(e+\frac12 u^2\right)_t+(up+q-uS)_x=0,\\
%v_t-u_x=0,
%\end{gather*}
%so the specific volume and the total energy both satisfy conservation laws of the form $U_t+f(U)_x=0$. 

We assume that the expressions $p=p(v,\theta,q,S)$ and $e=e(v,\theta,q,S)$ are unknown. 
Let $e^0$, $p^0$ and $s^0$ be the values of $p$ and $e$ at the equilibrium state $(v_\pm,\theta_\pm,q_\pm,S_\pm)=(1,1,0,0),$ i.e.,
\begin{equation*}
p^0=p(1,1,0,0),\quad\quad\quad e^0=e(1,1,0,0),\quad\quad\quad s^0=s(1,1,0,0).
\end{equation*}
Then by the Gibbs equation \eqref{Gibbs}, the function 
\begin{equation*}
\tilde{\eta }(v,u,s)=\hat{e}(v,s,q,S)+\frac12 u^2，
\end{equation*}
with $v,u,s$ as  independent variables, is convex. Here we regard the internal energy $e=\hat{e}(v,s,q,S)$ as a function of $(v,s,q,S).$ Therefore, the non-negative relative entropy 
\begin{equation*}
\eta=\left(e+\frac12 u^2\right)+p^0(v-1)-\left(s-s^0\right)-e^0
\end{equation*}
is obtained by subtracting the first two terms of the Taylor expansion of the convex function 
$\tilde{\eta}(v,u,s)$ from $\tilde{\eta}(v,u,s).$

The non-negative relative entropy $\eta$ satisfies 
\begin{equation}\label{eta-1}
\eta_t+\left(u(p-p^0)+q-uS\right)_x=-s_t
\end{equation}
by \eqref{HNS}. Since the dissipative part in above identity is solely due to the entropy $s_t,$ we calculate $s_t$ in the next step. 

Let $s=\bar{s}(v,e,q,S),$ then from \eqref{HNS}, \eqref{s_v}, \eqref{s_e} we have
\begin{align*}
s_t=&\frac{\partial \bar{s} }{\partial v}v_t+\frac{\partial \bar{s} }{\partial e}e_t+\frac{\partial \bar{s} }{\partial q}q_t+\frac{\partial \bar{s} }{\partial S}S_t \\
=&-\left(\frac{q}{\theta } \right)_x-\frac{q}{\tau_1}\frac{\partial \bar{s} }{\partial q}-\frac{S}{\tau_2}\frac{\partial \bar{s} }{\partial S}-\left(\frac{q}{\theta^2}+\frac{\kappa }{\tau_1v}\frac{\partial \bar{s} }{\partial q} \right)\theta_x+\left(\frac{S}{\theta}+\frac{\mu }{\tau_2v}\frac{\partial \bar{s} }{\partial S} \right)u_x.
\end{align*}
Substituting the above result into \eqref{eta-1}, one has
\begin{multline*}
\eta_t+\left(u(p-p^0)+\left(1-\frac1\theta\right)q-uS \right)_x+\left(-\frac{q}{\tau_1}\frac{\partial \bar{s} }{\partial q}-\frac{S}{\tau_2}\frac{\partial \bar{s} }{\partial S}\right)\\
=\left(\frac{q}{\theta^2}+\frac{\kappa }{\tau_1v}\frac{\partial \bar{s} }{\partial q} \right)\theta_x-\left(\frac{S}{\theta}+\frac{\mu }{\tau_2v}\frac{\partial \bar{s} }{\partial S} \right)u_x.
\end{multline*}
Here $\left(-\frac{q}{\tau_1}\frac{\partial \bar{s} }{\partial q}-\frac{S}{\tau_2}\frac{\partial \bar{s} }{\partial S}\right)$ is the dissipative part and the coefficients $\left(\frac{q}{\theta^2}+\frac{\kappa }{\tau_1v}\frac{\partial \bar{s} }{\partial q} \right) $ 
in front of the first derivative $\theta_x$ and $\left(\frac{S}{\theta}+\frac{\mu }{\tau_2v}\frac{\partial \bar{s} }{\partial S} \right)$ in front of the first derivative $u_x$ must be zero to ensure that the relative entropy $\eta$ is dissipative in the mathematical sense, i.e., 
\begin{equation}\label{s-cond}
\frac{\partial\bar{s}(v,e,q,S) }{\partial q} =-\frac{\tau_1vq}{\kappa\theta^2},\quad\quad\quad \frac{\partial\bar{s}(v,e,q,S) }{\partial q}=-\frac{\tau_2vS}{\mu\theta}.
\end{equation}
Under the assumption \eqref{s-cond}, the entropy $s$ satisfies
\begin{equation*}
s_t+\left(\frac{q}{\theta} \right)_x=\frac{vq^2}{\kappa\theta^2}+\frac{vS^2}{\mu\theta}.
\end{equation*}
and the  non-negative relative entropy $\eta$ satisfies 
\begin{equation*}
\eta_t+\left(u(p-p^0)+\left(1-\frac1\theta\right)q-uS \right)_x+\frac{vq^2}{\kappa\theta^2}+\frac{vS^2}{\mu\theta}=0.
\end{equation*}
Here $\frac{vq^2}{\kappa\theta^2}+\frac{vS^2}{\mu\theta}$ is the dissipative part.

Next, we look for expressions of $p$ and $e$ in terms of the independent variables $q$ and $S$  such that it satisfies assumption \eqref{s-cond}. From condition \eqref{s-cond}, we might guess
\begin{equation*}
s\sim \tilde{s}(v,e)+C_1(v,e)q^2+C_2(v,e)S^2.
\end{equation*}
However, this is not rigorous. The reason is that $\theta=\bar{\theta}(v,e,q,S)$ in the expressions for $\frac{\partial\bar{s} }{\partial q}$ and $\frac{\partial\bar{s} }{\partial S}$ in \eqref{s-cond} also depends on $q$ and $S$. Therefore we consider $(v,\theta,q,S)$ as independent variables. Since
\begin{equation*}
s=s(v,\theta,q,S)=\bar{s}(v,e(v,\theta,q,S),q,S),
\end{equation*}
from \eqref{s_e} and \eqref{s-cond}, we have
\begin{multline}\label{s_q-}
\frac{\partial s(v,\theta,q,S)}{\partial q} =\frac{\partial\bar{s}(v,e,q,S)}{\partial e} \frac{\partial e(v,\theta,q,S)}{\partial q}+\frac{\bar{s}(v,e,q,S)}{\partial q} \\
=\frac{1}{\theta } \frac{\partial e(v,\theta,q,S)}{\partial q}-\frac{\tau_1vq}{\kappa\theta^2}
\end{multline}
and
\begin{multline}\label{s_S-}
\frac{\partial s(v,\theta,q,S)}{\partial S} =\frac{\partial\bar{s}(v,e,q,S)}{\partial e} \frac{\partial e(v,\theta,q,S)}{\partial S}+\frac{\bar{s}(v,e,q,S)}{\partial S} \\
=\frac{1}{\theta } \frac{\partial e(v,\theta,q,S)}{\partial S}-\frac{\tau_2vS}{\mu\theta}.
\end{multline}
Thus, we conjecture that the terms involving $(q,S)$ in $e(v,\theta,q,S)$ are quadratic. Then from the transformation relations of the equations of state, the terms involving $(q,S)$ in $p(v,\theta,q,S)$ are also quadratic. We set
\begin{align}
e(v,\theta,q,S)=&\tilde{e}(v,\theta)+a(v,\theta)q^2+b(v,\theta)S^2, \label{assume-e}\\
p(v,\theta,q,S)=&\tilde{p}(v,\theta)+\alpha(v,\theta)q^2+\beta(v,\theta)S^2.\label{assume-p}
\end{align}
Substituting these into \eqref{s_q-} and \eqref{s_S-} gives
\begin{equation*}
\frac{\partial s(v,\theta,q,S)}{\partial q}=\left(\frac{2a(v,\theta)}{\theta}-\frac{\tau_1v}{\kappa\theta^2} \right)q
\end{equation*}
and 
\begin{equation*}
\frac{\partial s(v,\theta,q,S)}{\partial S}=\left(\frac{2b(v,\theta)}{\theta}-\frac{\tau_2v}{\mu\theta} \right)S.
\end{equation*}
Therefore we can assume the expression for $s(v,\theta,q,S)$ is
\begin{equation}\label{assume-s}
s(v,\theta,q,S)=\tilde{s}(v,\theta)+\left(\frac{a(v,\theta)}{\theta}-\frac{\tau_1v}{2\kappa\theta^2} \right)q^2+\left(\frac{b(v,\theta)}{\theta}-\frac{\tau_2v}{2\mu\theta} \right)S^2.
\end{equation}

Now we determine the coefficients $a(v,\theta),$ $b(v,\theta)$, $\alpha(v,\theta),$ $\beta(v,\theta)$ in three steps. Our main idea is to combine \eqref{assume-e}, \eqref{assume-p} and \eqref{assume-s} with the Gibbs equation \eqref{Gibbs}.

\begin{enumerate}
\item[(a).]  From Gibbs equation \eqref{Gibbs} we obtain 
$$\frac{\partial s(v,\theta,q,S)}{\partial \theta}=\frac1\theta\frac{\partial e(v,\theta,q,S)}{\partial \theta},$$
which, further combined with \eqref{assume-e} and \eqref{assume-s}, implies
\begin{multline*}
\frac{\partial \tilde{s}(v,\theta)}{\partial 
\theta} +\left(\frac{1}{\theta } \frac{\partial a(v,\theta)}{\partial\theta}-\frac{a(v,\theta) }{\theta^2 } +\frac{\tau_1v}{\kappa\theta^3} \right)q^2\\
+\left(\frac{1}{\theta } \frac{\partial b(v,\theta)}{\partial\theta}-\frac{b(v,\theta)}{ \theta^2} +\frac{\tau_2v}{2\mu\theta^2} \right)S^2 \\
=\frac{1}{\theta }\left(\frac{\partial \tilde{e}(v,\theta)}{\partial \theta }+\frac{\partial a(v,\theta)}{\partial\theta }q^2+\frac{\partial b(v,\theta)}{\partial \theta }S^2  \right).
\end{multline*}
Since $v,\theta,q,S$ are independent variables, the coefficients of $q^2$ and $S^2$ on both sides of the above equation are equal, i.e.,
\begin{equation}\label{a-b}
a(v,\theta)=\frac{\tau_1v}{\kappa\theta},\quad\quad b(v,\theta)=\frac{\tau_2v}{2\mu}.
\end{equation}
\item[(b).]  From Gibbs equation \eqref{Gibbs} we obtain
$$\frac{\partial s(v,\theta,q,S)}{\partial v}=\frac{\partial p(v,\theta,q,S)}{\partial \theta } ,$$
which, further combined with \eqref{assume-p} and \eqref{assume-s}, implies
\begin{multline*}
\frac{\partial \tilde{s}(v,\theta)}{\partial v}
+\left(\frac{1}{\theta }\frac{\partial a(v,\theta)}{\partial v}-\frac{\tau_1}{2\kappa\theta^2} \right)q^2+\left(\frac{1}{\theta } \frac{\partial b(v,\theta)}{\partial v}-\frac{\tau_2}{2\mu\theta} \right)S^2\\
=\frac{\partial \tilde{p}(v,\theta)}{\partial \theta } +\frac{\partial \alpha(v,\theta)}{\partial \theta  } q^2+\frac{\partial \beta(v,\theta)}{\partial \theta } S^2.
\end{multline*}
Similarly, equate the coefficients in front of $q^2$ and $S^2$, obtaining
\begin{equation}\label{alpha-beta_theta}
\frac{\partial \alpha(v,\theta)}{\partial \theta }=\frac{\tau_1}{2\kappa\theta^2},\quad\quad\quad \frac{\partial \beta(v,\theta)}{\partial \theta } =0.
\end{equation}
\item[(c).]  From \eqref{gibbs}, \eqref{assume-e} and \eqref{assume-p}, we have
\begin{multline*}
\frac{\partial \tilde{e}(v,\theta)}{\partial v } +\frac{\partial a(v,\theta)}{\partial v} q^2+\frac{\partial b(v,\theta)}{\partial v} S^2 \\
=\left(\theta\frac{\partial \tilde{p}(v,\theta)}{\partial \theta }-\tilde{p}(v,\theta) \right)+
\left(\theta\frac{\partial \alpha(v,\theta)}{\partial \theta }-\alpha(v,\theta) \right)q^2\\
+\left(\theta\frac{\partial \beta(v,\theta)}{\partial \theta }-\beta(v,\theta) \right)S^2.
\end{multline*}
Similarly, equate the coefficients in front of $q^2$ and $S^2$, obtaining
\begin{align*}
\alpha(v,\theta)=&\theta\frac{\partial \alpha(v,\theta)}{\partial \theta } -\frac{\partial a(v,\theta)}{\partial v} =-\frac{\tau_1}{2\kappa\theta} , \\
\beta(v,\theta)=&\theta\frac{\partial \beta(v,\theta)}{\partial \theta } -\frac{\partial b(v,\theta)}{\partial v} =-\frac{\tau_2}{2\mu}
\end{align*}
by \eqref{a-b} and \eqref{alpha-beta_theta}.
\end{enumerate}

From the above analysis, we obtain the following proposition.

\begin{proposition}
For the hyperbolic Navier-Stokes equations \eqref{HNS}, the pressure $p$, internal energy $e$, and entropy $s$ satisfy the relations
\begin{align}
p&=\tilde{p}(v,\theta)-\frac{\tau_1}{2\kappa\theta}q^2-\frac{\tau_2}{2\mu}S^2,   \\
e&=\tilde{e}(v,\theta)+\frac{\tau_1v}{\kappa\theta}q^2+\frac{\tau_2v}{2\mu}S^2,\\
s&=\tilde{s}(v,\theta) +\frac{\tau_1v}{2\kappa\theta^2}q^2.
\end{align}
Here $\tilde{p}(v,\theta), \tilde{e}(v,\theta),\tilde{s}(v,\theta)$ are the pressure, internal energy, and entropy of the Navier-Stokes equations when the relaxation parameters $\tau_1=\tau_2=0$. The relative entropy
\begin{align*}
\eta=&\left(e+\frac12 u^2\right)+p^0(v-1)-\left(s-s^0\right)-e^0\\ 
=&\tilde{\eta}(v,\theta)+\left(1-\frac{1}{2\theta} \right)\frac{\tau_1 v}{\kappa\theta}q^2+\frac{\tau_2v}{2\mu}S^2
\end{align*}
satisfies the equation
\begin{equation}\label{ }
\eta_t+\left(u(p-\tilde{p}(1,1))+\left(1-\frac1\theta\right)q-uS \right)_x+\frac{vq^2}{\kappa\theta^2}+\frac{vS^2}{\mu\theta}=0.
\end{equation}
Here $\tilde{\eta}(v,\theta)$ is the relative entropy when the relaxation parameters $\tau_1=\tau_2=0$.
\end{proposition}

\section*{Acknowledgments}
Qingsong Zhao was supported by the National Natural Science Foundation of China under Grant Number 12401281.

\end{document}